\documentclass[12pt,a4paper]{article}
\usepackage{authblk}
\usepackage[margin=2cm]{geometry}
\usepackage{t1enc}
\usepackage[utf8]{inputenc}
\usepackage{amsthm,amsmath,amssymb}
\usepackage{graphicx}
\usepackage{enumerate}
\usepackage{hyperref}
\usepackage{bm}
\usepackage{comment}
\usepackage{amsfonts}
\usepackage{graphicx,caption}
\usepackage{bm}
\usepackage{amsmath, amsthm, amssymb}
\usepackage{graphicx}
\usepackage{hyperref}
 \usepackage{relsize}
\usepackage{algpseudocode}

\usepackage{bbm}

\theoremstyle{plain}
\usepackage{amsthm}
\makeatletter
\newcommand{\newreptheorem}[2]{\newtheorem*{rep@#1}{\rep@title}\newenvironment{rep#1}[1]{\def\rep@title{#2 \ref*{##1}}\begin{rep@#1}}{\end{rep@#1}}}
\makeatother

\newtheorem{theorem}{Theorem}
\newtheorem*{theorem-non}{Theorem}
\newtheorem*{non-lemma}{Lemma}
\newtheorem{lemma}[theorem]{Lemma}
\newreptheorem{lemma}{Lemma}

\newtheorem{corr}[theorem]{Corollary}

\theoremstyle{definition}
\newtheorem{remark}[theorem]{Remark}

\DeclareMathOperator{\Tr}{Tr}

\DeclareMathOperator{\supp}{supp}

\DeclareMathOperator{\Var}{Var}

\DeclareMathOperator{\Lk}{Lk}

\DeclareMathOperator{\argmin}{arg\,min}

\DeclareMathOperator{\Image}{im}

\begin{document}
\title{Coboundary expansion for the union of determinantal hypertrees}
\author{Andr\'as M\'esz\'aros\thanks{HUN-REN Alfr\'ed R\'enyi Institute of Mathematics, Budapest, Hungary,\\ {\tt meszaros@renyi.hu}
}}
\date{}
\maketitle
\begin{abstract}
We prove that for any large enough constant $k$, the union of $k$ independent $d$-dimensional determinantal hypertrees is a coboundary expander with high probability.
\end{abstract}

\section{Introduction}

\emph{Determinantal hypertrees} are natural higher dimensional analogues of a uniform random spanning tree of a  complete graph. A $d$-dimensional simplicial complex $T$ on the vertex set $[n]=\{1,2,\dots,n\}$ is called a ($d$-dimensional) hypertree, if
\begin{enumerate}[\hspace{30pt}(a)]
    \item $T$ has complete $d-1$-skeleton;
    \item The number of $d$-faces of $K$ is ${n-1}\choose{d}$;
    \item The integral homology group $H_{d-1}(T)$ is finite.
\end{enumerate}

Kalai's generalization of Cayley's formula \cite{kalai1983enumeration} states that
\[\sum |H_{d-1}(T)|^2=n^{{n-2}\choose {d}},\]
where the summation is over all the  hypertrees $T$ on the vertex set $[n]$. This formula suggests that the natural probability measure on the set of  hypertrees is the one where the probability assigned to a hypertree $T$ is \begin{equation}\label{measuredef}
    \frac{|H_{d-1}(T)|^2}{n^{{n-2}\choose {d}}}.
\end{equation}
It turns out that this measure is a determinantal probability measure \cite{lyons2003determinantal,hough2006determinantal}. More specifically, it is the determinantal measure corresponding to the orthogonal projection to the subspace of coboundaries $B^d(\Delta_{[n]},\mathbb{R})$, where $\Delta_{[n]}$ is the simplex on the vertex set $[n]$. General random determinantal complexes were investigated by Lyons \cite{lyons2009random}. For some recent results on determinantal hypertrees, see \cite{kahle2022topology,meszaros2022local,werf2022determinantal,linial2019enumeration,meszaros20242}.

Kahle and Newman \cite{kahle2022topology} asked about the expansion properties of determinantal hypertrees. For graphs, there are several equivalent definitions of expansion. Each of these suggests an extension to higher dimensional simplicial complexes. However, the arising notions are no longer equivalent. Thus, in higher dimensions, we have several notions of expansion \cite{lubotzky2018high}. In this paper, we consider the notion of \emph{coboundary expansion}. This notion was (implicitly) defined by Linial-Meshulam \cite{linial2006homological}  and Gromov \cite{gromov2010singularities}. Gromov proved that coboundary expansion implies a certain overlapping property of continuous images of the complex. Other applications include coding theory \cite{codes1,codes2,codes3} and property testing \cite{kaufman2014high}.  

There are slight variations of the definition of coboundary expansion, we follow \cite{lubotzky2018high}. First, we need to review a few notions from cohomology theory. Let $\mathbb{F}$ be a field, for us only the cases $\mathbb{F}=\mathbb{R}$ or $\mathbb{F}_2$ will be important. Given a $d$-dimensional simplicial complex $K$, let $K(i)$ be the set of $i$ dimensional faces of $K$, and let $C^i(K,\mathbb{F})$ be the $\mathbb{F}$-vector space of functions from $K(i)$ to $\mathbb{F}$. The coboundary maps
\[\delta_{i,\mathbb{F}}:C^{i}(K,\mathbb{F})\to C^{i+1}(K,\mathbb{F}) \]
are defined the usual way, see for example \cite{lubotzky2018high}. If the complex $K$ is not clear from the context, we write $\delta_{i,\mathbb{F}}^K$. The space of $i$-coboundaries is defined as $B^i(K,\mathbb{F})=\Image \delta_{i-1,\mathbb{F}}$.

Next we define a norm on $C^i(K,\mathbb{F}_2)$. We first define the weight $w(\sigma)$ of an $i$-dimensional face $\sigma$ as the number of top dimensional faces containing $\sigma$ divided by a normalizing factor, that is,
\[w(\sigma)=\frac{1}{{{d+1}\choose{i +1}}|K(d)|}|\{\tau\in K(d)\,:\,\sigma\subset \tau\}|.\]
The normalizing factor is chosen such that the total weight of the $i$-dimensional faces is $1$.

Then, for $f\in C^i(K,\mathbb{F}_2)$, we define
\[\|f\|=\sum_{\sigma\in\supp f} w(\sigma).\]
We write $\|\cdot\|_K$, when $K$ is not clear from the context.

We can also extend this norm to cosets of $B^i(K,\mathbb{F}_2)$ by setting
\[\|f+B^i(K,\mathbb{F}_2)\|=\min_{g\in f+B^i(K,\mathbb{F}_2)} \|g\|.\]

Next, for $0\le i\le d-1$, we define the $i$th coboundary expansion constant of $K$ as
\begin{equation}\label{hidef}
    \mathfrak{h}_i(K)=\min_{f\in C^i(K,\mathbb{F}_2)\setminus B^i(K,\mathbb{F}_2)} \frac{\|\delta_i f\|}{\|f+B^i(K,\mathbb{F}_2)\|}.
\end{equation}
Finally, the coboundary expansion constant of $K$ is defined as
\[\mathfrak{h}(K)=\min_{i=0}^{d-1} \mathfrak{h}_i(K).\]

Now we are able to state our main theorem. Let us fix a dimension $d\ge 1$.  Let $T_{n,d}$ be a $d$-dimensional determinantal hypertree on the vertex set $[n]$ distributed according to \eqref{measuredef}. Let $k$ be a sufficiently large constant depending on $d$, and let $T_{n,d}^1,T_{n,d}^2,\dots,T_{n,d}^k$ be i.i.d. copies of~$T_{n,d}$, and let
\[K_{n,d,k}=\cup_{i=1}^k T_{n,d}^i.\]

\begin{theorem}\label{thm1}
There are $k_d$ and $\zeta_d>0$ such that for any constant $k\ge k_d$, we have
\[\lim_{n\to \infty} \mathbb{P}(\mathfrak{h}(K_{n,d,k})\ge \zeta_d)=1.\]

\end{theorem}

The spectral expansion of $K_{n,d,k}$ was proved by Vander Werf \cite{werf2022determinantal}  for $k=\delta \log n$, where $\delta$ is an arbitrary positive constant. The one dimensional analogue of our result was proved by Goyal, Rademacher and Vempala~\cite{goyal2009expanders}. Our theorem provides an example of a coboundary expander of bounded average upper degree. Note that there are coboundary expanders with bounded maximum upper degree, in dimension~$2$ these were constructed Lubotzky and  Meshulam~\cite{lubotzky2015random}, and for all dimensions by Lubotzky, Luria and Rosenthal   \cite{lubotzky2019random}.

It is clear from the definitions that $\mathfrak{h}_{d-1}(K_{n,d,k})>0$ if and only if ${H^{d-1}(K_{n,d,k},\mathbb{F}_2)=\{0\}}$. Also note that $H^{d-1}(K_{n,d,k},\mathbb{F}_2)\cong H_{d-1}(K_{n,d,k},\mathbb{F}_2)$. Thus, Theorem~\ref{thm1} has the following corollary.

\begin{corr}\label{corH}
Let $k_d$ be the constant provided by Theorem~\ref{thm1}. For any constant $k\ge k_d$, we have
\[\lim_{n\to \infty} \mathbb{P}\left(H_{d-1}(K_{n,d,k},\mathbb{F}_2)=\{0\}\right)=1.\]
\end{corr}

It is a natural question what is the smallest possible constant $k_d$ in Theorem~\ref{thm1}. In two dimensions, we can see that  $k_2=1$ is not sufficient by combining Corollary~\ref{corH} and the result of the author \cite{meszaros20242} that 
\[\liminf_{n\to\infty} \mathbb{P}(H_{1}(T_{n,2},\mathbb{F}_2)\neq \{0\})>0.\]
We expect the same in higher dimensions. On the other hand, we believe that $k_d=2$ should already work. Our proof gives a bound on $k_d$ which is at least doubly exponential in $d$, see Remark~\ref{remarkdouble2}. 

Note that we can sample determinantal hypertrees in polynomial time. In fact, this is true for any determinantal measure \cite{hough2006determinantal}. For $d=1$, the Aldous-Broder algorithm \cite{aldous1990random,broder1989generating} and Wilson's algorithm \cite{wilson1996generating} provide us an even more efficient way of sampling. It would be nice to generalize these algorithms for higher dimensions.

Kahle and Newman \cite{kahle2022topology} conjectured that the $p$-torsion of the first homology group of a $2$-dimensional determinantal hypertree asymptoticly Cohen-Lenstra distributed, see also~\cite{kahle2020cohen} for related conjectures. By the results of the author~\cite{meszaros20242}, we know that this conjecture is not true for $p=2$. The conjecture remains open for $p>2$. The survey of Wood \cite{wood2023probability} provides a good overview of the Cohen-Lenstra heuristics. For a simplified model motivated by determinantal hypertrees, the Cohen-Lenstra limiting distribution was established by the author \cite{meszaros2023cohen} for $p\ge 5$. The Cohen-Lenstra limiting distribution is usually proved by calculating the limiting moments and using the results of Wood~\cite{wood2017distribution}. In our context, this means that for any finite abelian $p$-group $G$, we need to calculate the limit of the expected size of the first cohomology group $H^1(T_{n,2},G)$, that is, the expected number of cosets $f+B^1(T_{n,2},G)$ such that $\delta_{1}f=0$. Previous papers \cite{wood2017distribution,wood2019random,meszaros2020distribution} in this topic suggest that we need to handle typical and non-typical cosets separately for some notion of typicality, such that the contribution  of non-typical cosets is negligible. Our hope is that our method can be refined to show that the contribution  of non-typical cosets is indeed negligible. Handling the typical cosets seems more challenging.

\subsection{Outline of the proof}

The key ingredient of the proof is the  local to global criterion of Evra and Kaufman~\cite{evra2016bounded}, see Section~\ref{secevrakaufman}. This criterion says that for a simplicial
complex, if all the links are coboundary expanders and the complex
and all its links are skeleton expanders, then small minimal cocycles expand. To apply this criterion, we need a good understanding of the links of determinantal hypertrees. Thus, in Section~\ref{secgendet}, we introduce generalized determinantal hypertrees and show that any link of a generalized determinantal hypertree is again a generalized determinantal hypertree as it was observed by Vander Werf~\cite{werf2022determinantal}. To be able to use induction, we will prove a slightly stronger version of Theorem~\ref{thm1}, where we take the union of generalized determinantal hypertrees, see Theorem~\ref{thmstronger}.

In Section~\ref{secskel}, we prove that with high probability  the union of $k$ generalized determinantal hypertrees is a skeleton expander provided that $k$ is large enough.

In Section~\ref{secfinal}, we use induction on $d$ to show that the union of $k$ generalized determinantal hypertrees is a coboundary expander for all large enough $k$. The Evra-Kaufman criterion gives the expansion of small minimal cocycles. The expansion of large minimal cocycles is obtained by combining concentration inequalities with the fact that complete complexes are expanders.

\medskip

\textbf{Acknowledgement:} The author is grateful to the anonymous referees for their useful comments.
The author was supported by the NSERC discovery grant of B\'alint Vir\'ag and the KKP 139502 project.

\section{Preliminaries on coboundary expansion}

\subsection{A few notations}
Note that any $f\in C^i(K,\mathbb{F}_2)$ is uniquely determined by its support. Given $A\subset K(i)$, let $f_A\in C^{i}(K,\mathbb{F}_2)$ be its characteristic vector. We use $A$ and $f_A$ interchangeably. Thus,
we use $\|A\|$ to denote $\|f_A\|$, and $\delta_{i,\mathbb{F}_2} A$ to denote $\supp \delta_{i,\mathbb{F}_2} f_A$.

For $\sigma\in K$, the link of $\sigma$ is defined as
\[\Lk(\sigma,K)=\{\tau\subset K(0)\setminus \sigma\,:\, \tau\cup\sigma\in K \}.\]

We use the notation ${{[n]}\choose{d}}$ for $\{\sigma\subset [n]\,:\,|\sigma|=d\}$.

\subsection{Minimal cochains}

An $f\in  C^i(K,\mathbb{F}_2)$ is called minimal if
\[\|f\|=\min_{g\in f+B^i(K,\mathbb{F}_2)} \|g\|.\]
An $A\subset C^i(K,\mathbb{F}_2)$ is called minimal if $f_A$ is minimal. If we want to emphasize the complex $K$, we say that $A$ is $K$-minimal.

Equation \eqref{hidef} can be rephrased as
\begin{equation}\label{hiseconddef}
    \mathfrak{h}_i(K)=\min_{\substack{\emptyset \neq A\in C^i(K,\mathbb{F}_2)\\A\text{ is minimal }}} \frac{\|\delta_i A\|}{\|A\|}.
\end{equation}

\subsection{The local to global criteria of Evra and Kaufman}\label{secevrakaufman}

Let $\alpha>0$. We say that the complex $K$ is an $\alpha$-skeleton expander, if for all $A\subset K(0)$, we have
\[\|E(A,A)\|\le 4(\|A\|^2+\alpha \|A\|),\]
where $E(A,A)\subset K(1)$ are the edges in $K$ with both vertices in $A$.

\begin{theorem}(Evra, Kaufman \cite[Theorem 3.2]{evra2016bounded})\label{evrakaufman}
For any $d\ge 1$ and $\beta>0$, there are positive constants $\alpha=\alpha(d,\beta)$, $\bar{\varepsilon}=\bar{\varepsilon}(d,\beta)$ and $\bar{\mu}=\bar{\mu}(d,\beta)$ with the following property. Assume that $K$ is a $d$-dimensional complex satisfying that
\begin{enumerate}[(a)]
    \item For all $\sigma \in K$ such that $1\le |\sigma|\le d-1$, we have
    \[\mathfrak{h}(\Lk(\sigma,K))\ge \beta.\]
    \item For all $\sigma \in K$ such that $0\le |\sigma|\le d-1$, the link $\Lk(\sigma,K)$ is an $\alpha$-skeleton expander.
\end{enumerate}
Then for any $0\le i\le d-1$ and a minimal $A\subset K(i)$ such that $\|A\|\le \bar{\mu}$, we have

\[\|\delta_{i,\mathbb{F}_2} A\|\ge \bar{\varepsilon}\|A\|. \]

\end{theorem}
\begin{remark}
    In the original statement of \cite[Theorem 3.2]{evra2016bounded}, $A$ is required to be locally minimal. However, since every minimal cochain is locally minimal \cite[Lemma 2.18]{evra2016bounded}, the version given in Theorem~\ref{evrakaufman} is also true.
\end{remark}

\subsection{Expansion of the complete simplicial complex}

\begin{lemma}[Meshulam and Wallach \cite{meshulam2009homological}]\label{Deltaexpansion}
Consider the complete simplicial complex $\Delta_{[n]}$ on $n$ vertices. Let $A\subset {{[n]}\choose {i+1}}$ be minimal. Then
\[\left|\delta_{i,\mathbb{F}_2}^{\Delta_{[n]}} A\right|\ge \frac{|A|n}{i+2}.\]
\end{lemma}

\begin{remark}
    Assume that $K$ is a pure $d$-dimensional simplicial complex on $[n]$ with complete $d-1$-skeleton and maximum upper degree at most $D$. Then for all $0\le i\le d-1$ and $A\subset {{[n]}\choose {i+1}}$, we have
    \[\frac{|A|{{n-i-1}\choose{d-i-1}}}{(d-i){{d+1}\choose{i+1}}|K(d)|}\le \|A\|\le \frac{|A|{{n-i-1}\choose{d-i-1}}D}{(d-i){{d+1}\choose{i+1}}|K(d)|}.\]
    Combining this with Lemma~\ref{Deltaexpansion}, we see that for all $0\le i\le d-2$, we have
    \[\mathfrak{h}_i(K)\ge \frac{1}{D}.\]
    Thus, if we want to understand the coboundary expansion of complexes like above, we can restrict our attention to $\mathfrak{h}_{d-1}$. 

    Although $K_{n,d,k}$ is a pure $d$-dimension complex with complete $d-1$-skeleton, we do not have a uniform bound on maximum upper degree of $K_{n,d,k}$. Thus, in this case comparing $\|A\|$ and $|A|$ requires some additional effort, see Lemma~\ref{NormVSSize}. Therefore, bounding $\mathfrak{h}_i(K_{n,d,k})$ is a nontrivial problem even for $i<d-1$.
\end{remark}

\section{Preliminaries on determinantal processes}

\subsection{The definition of a discrete determinantal process}\label{secdetdef}

An $m\times m$ matrix  $Q$ over the reals is called a positive contraction, if $Q=Q^T$ and the spectrum of $Q$ is contained in $[0,1]$. Equivalently, $Q$ is a positive contraction if and only if it is a principal minor of some orthogonal projection matrix, as it was shown in \cite[Section 8]{lyons2003determinantal}. Assume that the rows and columns of $Q$ are indexed with the set $E$. Given a subset $A$ of $E$, let $Q_A$ be the submatrix of $Q$ determined by the rows and columns with index in $A$. There is a random subset $X$ of $E$ with the property that for all $A\subset E$, we have
\[\mathbb{P}(A\subset X)=\det Q_A.\]
The law of $X$ is uniquely determined by the property above. We call $X$ the determinantal process corresponding to $Q$. See \cite{lyons2003determinantal,hough2006determinantal} for more information on determinantal measures.

Given a positive integer $k$, let $X^{1},X^{2},\dots,X^k$ be $k$ independent copies of $X$, and let
\[X^{\Sigma }=\cup_{i=1}^k \{i\}\times X^i\subset [k]\times E,\]
and let
\[X^{\cup }=\cup_{i=1}^k  X^i\subset  E.\]

Note that $X^{\Sigma }$ is also a determinantal process corresponding to the matrix $Q^\Sigma$, which is the direct sum of $k$ copies of the matrix $Q$.

For $A\subset E$, we define $A^\Sigma=[k]\times A$.

Note that $X^\Sigma,X^\cup$ and $A^\Sigma$ depend on $k$, but we suppress this dependence in the notations.

\subsection{Negative associations}

Determinantal processes have negative associations \cite{lyons2003determinantal}. For us only a special case of this fact will be important, which easily follows from Hadamard's inequality: For any $A\subset E$, we have
\begin{equation}\label{basicHadamard}\mathbb{P}(A\subset X)=\det Q_A\le \prod_{a\in A} Q(a,a).
\end{equation}

As a corollary, we obtain the following lemma.

\begin{lemma}\label{unionHadamard}
Let $A\subset E$. Then
\[\mathbb{P}(A\subset X^{\cup})\le k^{|A|}\prod_{a\in A}  Q(a,a).\]
\end{lemma}
\begin{proof}
Given a vector $f\in [k]^A$, let $A_f=\{(f(a),a)\,:\, a\in A\}$. By the union bound and \eqref{basicHadamard}, we have
\begin{align*}
\mathbb{P}(A\subset X^{\cup})&=\mathbb{P}(A_f\subset X^{\Sigma}\text{ for some }f\in [k]^A)\\&\le \sum_{f\in [k]^A} \mathbb{P}(A_f\subset X^{\Sigma})\\&\le \sum_{f\in [k]^A} \prod_{a\in A} Q^\Sigma((f(a),a),(f(a),a))\\&=k^{|A|} \prod_{a\in A}  Q(a,a).
\end{align*}
\end{proof}

\subsection{Concentration results for determinantal processes}

\begin{lemma}\cite[Theorem 7]{hough2006determinantal}\label{bernoillisum}
Let $A$ be a $k$-element subset of $E$, and let $\lambda_1,\lambda_2,\dots,\lambda_k$ be the eigenvalues of $Q_A$. Then $X\cap A$ has the same distribution as $\sum_{i=1}^k B_i$, where $B_1,B_2,\dots,B_k$ are independent Bernoulli random variables such that $\mathbb{P}(B_i=1)=\lambda_i$.
\end{lemma}

\begin{lemma}\label{berstein}
Let $0\le \varepsilon\le 1$, and let $A$ be a $k$-element subset of $E$. Then
\[\mathbb{P}\left(\Big||X\cap A|-\mathbb{E}|X\cap A|\Big|\ge \varepsilon \mathbb{E}|X\cap A|\right)\le 2\exp\left(-\frac{\varepsilon^2}{4}\mathbb{E}|X\cap A|\right).\]
\end{lemma}
\begin{proof}
      Note that $|B_{i}-\mathbb{E}B_{i}|\le 1$, and
      \[ \sum_{i=1}^k \Var(B_{i})= \sum_{i=1}^k \lambda_i(1-\lambda_i)\le \sum_{i=1}^k \lambda_i= \Tr Q_A=\mathbb{E} |X\cap A|.\]
      Thus, the statement follows by combining Lemma~\ref{bernoillisum} with Bernstein's inequality.
\end{proof}

\section{Generalized determinantal hypertrees}\label{secgendet}

Let $P_{n,d}$ be the orthogonal projection from $C^d(\Delta_{[n]},\mathbb{R})$ to $B^d(\Delta_{[n]},\mathbb{R})$. Here $C^d(\Delta_{[n]},\mathbb{R})$ is endowed with the standard inner product
\[\langle f,g\rangle=\sum_{\sigma\in {{[n]}\choose{d+1}}} f(\sigma)g(\sigma).\]
Next we give a more explicit description of $P_{n,d}$ based on \cite{meszaros2022local}. For $d\ge 1$, let $I_{n,d}$ be a matrix indexed by ${{[n]}\choose {d}}\times {{[n]}\choose{d+1}}$ defined as follows. Let $\sigma=\{x_0,x_1,\dots,x_d\}\subset [n]$ such that $x_0<x_1<\dots<x_d$. For a $\tau\in {{[n]}\choose {d}}$, we set
\[I_{n,d}(\tau,\sigma)=\begin{cases}
(-1)^i&\text{if }\tau=\sigma\setminus\{x_i\},\\
0&\text{otherwise.}
\end{cases}
\]  
Note that $I_{n,d}^T$ is just the matrix of the coboundary map $\delta_{d-1,\mathbb{R}}^{\Delta_{[n]}}$.

For two sets $\sigma,\tau\in {{[n]}\choose {d+1}}$ such that $|\sigma\cap \tau|=d$, we introduce the notation \[J(\sigma,\tau)=I_{n,d}(\sigma\cap \tau,\sigma)\cdot I_{n,d}(\sigma\cap \tau,\tau).\]

\begin{lemma}\cite[Lemma 2.7]{meszaros2022local} We have
\[P_{n,d}=\frac{1}n I_{n,d}^T I_{n,d}.\]
Consequently,
\[P_{n,d}(\sigma,\tau)=\begin{cases}
\frac{d+1}{n}&\text{if $\sigma=\tau$,}\\
\frac{1}n J(\sigma,\tau)&\text{if $|\sigma\cap \tau|=d$,}\\
0&\text{otherwise.}
\end{cases}
\]
\end{lemma}

The next lemma is a special case of \cite[Proposition 3.1]{lyons2009random}.
\begin{lemma}
The random hypertree $T_{n,d}$ is the determinantal process corresponding to $P_{n,d}$.
\end{lemma}

Note that here with a slight abuse, we identified $T_{n,d}$ with the set of its top dimensional faces $T_{n,d}(d)$.

Next we slightly generalize the notions above to be able to describe the distribution of the links of a determinantal hypertree.  For $0\le \ell<n$, let $P_{n,d,\ell}$ be the matrix where the rows and columns are indexed with ${{[n]}\choose{d+1}}$, and for $\sigma,\tau\in {{[n]}\choose{d+1}}$, we have

\[P_{n,d,\ell}(\sigma,\tau)=\begin{cases}
\frac{d+1+\ell}{n+\ell}&\text{if $\sigma=\tau$,}\\
\frac{1}{n+\ell} J(\sigma,\tau)&\text{if $|\sigma\cap \tau|=d$,}\\
0&\text{otherwise.}
\end{cases}
\]

Note that
\begin{equation}\label{eq:perc}
    P_{n,d,\ell}=\frac{n}{n+\ell} P_{n,d}+\frac{\ell}{n+\ell}I
\end{equation}
in particular, $P_{n,d,0}=P_{n,d}$.

The proof of the next lemma is also straightforward.
\begin{lemma}\label{lemmasubmatrix}
Consider the principal submatrix $Q$ of $P_{n,d}$ corresponding to the columns which are indexed with sets $\sigma$ such that $\{n-\ell+1,n-\ell+2,\dots,n\}\subset \sigma$. Then $Q=P_{n-\ell,d-\ell,\ell}$, after we identify the column of $P_{n-\ell,d-\ell}$ corresponding to $\tau$ with the column of $Q$ corresponding to $\tau\cup \{n-\ell+1,n-\ell+2,\dots,n\}$.

More generally, if we consider the principal submatrix $Q$ of $P_{n,d,\ell_1}$ corresponding to the columns which are indexed with sets $\sigma$ such that $\{n-\ell_2+1,n-\ell_2+2,\dots,n\}\subset \sigma$. Then \break $Q=P_{n-\ell_2,d-\ell_2,\ell_1+\ell_2}$, after the same identification as above.

\end{lemma}

Let $T_{n,d,\ell}$ be the determinantal process corresponding to $P_{n,d,\ell}$. We call $T_{n,d,\ell}$ a generalized determinantal hypertree. Note that by Lemma~\ref{lemmasubmatrix}, the matrix $P_{n,d,\ell}$ is a principal submatrix of an orthogonal projection matrix, so $P_{n,d,\ell}$ is a positive contraction. Thus, the definition of $T_{n,d,\ell}$ makes sense.

Using \eqref{eq:perc}, one can see that we can also obtain $T_{n,d,\ell}$ from $T_{n,d}$ by independently adding each missing $d$-dimensional face with probability $\frac{\ell}{n+\ell}$, see also \cite{werf2022determinantal}.

Using Lemma~\ref{lemmasubmatrix} and symmetry, we obtain the following lemma, which was also observed in \cite{werf2022determinantal}.

\begin{lemma}\label{Lemmalinks}
For $0\le \ell_2\le d-1$, let $\sigma\in {{[n]}\choose{\ell_2}}$. Then after relabeling the vertices, the link $\Lk(\sigma,T_{n,d,\ell_1})$ has the same distribution as $T_{n-\ell_2,d-\ell_2,\ell_1+\ell_2}$.
\end{lemma}

\begin{lemma}\label{pure}
The complex $T_{n,d,\ell}$ is always a pure $d$-dimensional complex. In other words, for $\sigma\in {{[n]}\choose{d}}$, there is always a $d$-dimensional face of $T_{n,d,\ell}$ containing $\sigma$.
\end{lemma}
\begin{proof}
Since the links of a pure simplicial complex are also pure, using Lemma~\ref{Lemmalinks}, it is enough to prove the statement for $\ell=0$. Let $\sigma\in {{[n]}\choose{d}}$, let $f_\sigma\in C^{d-1}(\Delta_{[n]},\mathbb{R})$ be the characteristic vector of $\sigma$. Then $\delta_{d-1,\mathbb{R}} f_\sigma$ is supported on
\[A=\{\sigma\cup \{i\}\,:\,u\in [n]\setminus \sigma\}.\]
Since $\delta_{d-1,\mathbb{R}} f_\sigma\in B^{d}(\Delta_{[n]},\mathbb{R})$, it is an eigenvector of $P_{n,d}$ with eigenvalue $1$. Then restricting $\delta_{d-1,\mathbb{R}} f_\sigma$ to $A$, we see that $1$ is an eigenvalue of the submatrix $(P_{n,d})_A$. Thus, by Lemma~\ref{bernoillisum}, we have $|T_{n,d}\cap A|\ge 1$.
\end{proof}

Finally,  for $k \ge 1$, similarly as in Section~\ref{secdetdef}, we define
\[T_{n,d,\ell}^\Sigma=\{ \{i\}\times \sigma\,:\,1\le i\le k, \sigma\in T_{n,d,\ell}^i\}\qquad\text{ and }\qquad T_{n,d,\ell}^\cup=\cup_{i=1}^k T_{n,d,\ell}^i,\]
where $T_{n,d,\ell}^1,\dots,T_{n,d,\ell}^k$ are i.i.d. copies of $T_{n,d,\ell}$.
\section{Skeleton expansion}\label{secskel}

Our goal in this section is to prove the following result:
\begin{lemma}\label{lemmaskeleton}
For any fixed $\ell\ge 0$, $d\ge 1$, $s>0$ and $\alpha>0$, there is a $k_0$ such that for all $k>k_0$ the following holds: $T_{n,d,\ell}^\cup$ is an $\alpha$-skeleton expander with probability $O(n^{-s})$.
\end{lemma}

Throughout this section we think of $\ell,d,s$ and $\alpha$ as fixed constants, and the hidden constant in the $O(\cdot)$ notation is allowed to depend on $\ell,d,s,\alpha$ and $k$.

In this section, we use the notations $K_n^\cup=T_{n,d,\ell}^\cup$ and $K_n^{\Sigma}=T_{n,d,\ell}^{\Sigma}$.

Let $\alpha_0=\frac{\alpha}{(d+1)!}$.

Clearly, we may assume that $\alpha_0<1$.
\begin{lemma}\label{SigmaVSCup}
For any constant $k$, with probability at least $1-O(n^{-s})$ the following event occurs: For all $\sigma\in {{[n]}\choose{d}}$, we have
\[|K_n^\cup\cap C_\sigma|> |K_n^\Sigma\cap C_\sigma^\Sigma|-s-d,\]
where
\[C_\sigma=\{\sigma\cup \{u\}\,:\, u\in [n]\setminus \sigma\},\]
and as before $C_\sigma^\Sigma=[k]\times C_\sigma.$
\end{lemma}
\begin{proof}
Let $\sigma\in {{[n]}\choose{d}}$ and consider the event that $|K_n^\cup\cap C_\sigma|\le |K_n^\Sigma\cap C_\sigma^\Sigma|-s-d$. On this event, we have a subset $E$ of $C_\sigma$ such that $1\le |E|\le s+d$ and $ |K_n^\Sigma\cap E^\Sigma|\ge |E|+s+d$.

For any $E$ as above, we have
\begin{align*}
\mathbb{P}( |K_n^\Sigma\cap E^\Sigma|\ge |E|+s+d)&\le \sum_{\substack{F\subset [k]\times E\\|F|=|E|+s+d}} \mathbb{P}(F\subset K_n^\Sigma)\\&\le (k|E|)^{|E|+s+d} \left(\frac{d+1+\ell}{n+\ell}\right)^{|E|+s+d}=O(n^{-(|E|+s+d)}),
\end{align*}
where we used \eqref{basicHadamard} to estimate $\mathbb{P}(F\subset K_n^\Sigma)$.

Therefore,
\begin{align*}
    \mathbb{P}\left(|K_n^\cup\cap C_\sigma|\le |K_n^{\Sigma}\cap C_\sigma^\Sigma|-s-d\right)&\le \sum_{h=1}^{s+d} \sum_{\substack{E\subset C_\sigma\\|E|=h}} \mathbb{P}\left(|K_n^\Sigma\cap E^\Sigma|\ge |E|+s+d\right)\\
    &\le \sum_{h=1}^{s+d} n^h O(n^{-(h+s+d)})=O(n^{-s-d}).
\end{align*}

Using the union bound over the choice of $\sigma$, the statement follows.
\end{proof}

For $A\subset [n]$, let
\[E^d_{\ge 2}(A)=\left\{\sigma\in {{n}\choose{d+1}}\,:\,|\sigma\cap A|\ge 2\right\}.\]
Note that
\[|E^d_{\ge 2}(A)|\le |A|^2n^{d-1}.\]
\begin{lemma}\label{Lemma15}
For any $A\subset [n]$ such that $m=|A|<\frac{\alpha_0 n}{e(d+1+\ell)}$, we have
\[\mathbb{P}(|E^d_{\ge 2}(A)\cap K_n^\cup|\ge k\alpha_0 m n^{d-1})\le \left(\frac{e(d+1+\ell)m}{n\alpha_0 }\right)^{k\alpha_0 m}.\]
\end{lemma}
\begin{proof}
Let $t=\lceil k\alpha_0 mn^{d-1} \rceil$.
Using Lemma~\ref{unionHadamard}, we have
\begin{multline*}\mathbb{P}(|E^d_{\ge 2}(A)\cap K_n^\cup|\ge t)\le \sum_{\substack{E\subset E_{\ge 2}^d(A)\\|E|=t}} \mathbb{P}(E\subset K_n^\cup)\le {{m^2n^{d-1}}\choose t} \left(\frac{k(d+1+\ell)}{n+\ell}\right)^t\\\le \left(\frac{em^2n^{d-1}k(d+1+\ell)}{tn}\right)^t\le \left(\frac{e(d+1+\ell)m}{n\alpha_0 }\right)^{k\alpha_0 mn^{d-1}} \le \left(\frac{e(d+1+\ell)m}{n\alpha_0 }\right)^{k\alpha_0 m},
\end{multline*}
where we used that $\frac{e(d+1+\ell)m}{n\alpha_0 }<1$.
\end{proof}

Let $\mu=\frac{\alpha_0^2}{2e^3(d+1+\ell)^2}$.

\begin{lemma}\label{Lemma16}
Assume that $k$ is large enough. Then with probability at least $1-O(n^{-s})$, the following holds: for all $A\subset [n]$ such that $0<|A|<\mu n$, we have
\[|E^d_{\ge 2}(A)\cap K_n^\cup|< k\alpha_0 |A|n^{d-1}.\]
\end{lemma}
\begin{proof}
Combining Lemma~\ref{Lemma15} with the union bound, we obtain that the complement of the event above has probability at most
\begin{align*}\sum_{m=1}^{\lfloor\mu n\rfloor} \sum_{A\in {{[n]}\choose{m}}}& |\mathbb{P}(E^d_{\ge 2}(A)\cap K_n^\cup|\ge k\alpha_0 |A|n^{d-1})\\&\le \sum_{m=1}^{\lfloor\mu n\rfloor} {{n}\choose {m}} \left(\frac{e(d+1+\ell)m}{n\alpha_0 }\right)^{k\alpha_0 m}\\&\le \sum_{m=1}^{\lfloor\mu n\rfloor} \left(\frac{ne}{m}\right)^m  \left(\frac{e(d+1+\ell)m}{n\alpha_0 }\right)^{k\alpha_0 m}\\&=  \sum_{m=1}^{\lfloor\mu n\rfloor} \left(\frac{e^2(d+1+\ell)}{\alpha_0}\right)^m  \left(\frac{e(d+1+\ell)m}{n\alpha_0 }\right)^{(k\alpha_0-1) m}.\end{align*}

Let us choose $k$ such that $k\alpha_0-1>2s+1$. Recalling our assumption that $\alpha_0<1$, we see that $\frac{e^2(d+1+\ell)}{\alpha_0}>1$ and $\frac{e(d+1+\ell)m}{n\alpha_0 }< 1$ for all $1\le m\le \mu n$. Thus, the  sum above is bounded by
\[\sum_{m=1}^{\lfloor\mu n\rfloor} \left(\frac{e^2(d+1+\ell)}{\alpha_0}\right)^{(2s+1)m}  \left(\frac{e(d+1+\ell)m}{n\alpha_0 }\right)^{(2s+1) m}=\sum_{m=1}^{\lfloor\mu n\rfloor} \left(\frac{m}{2\mu n}\right)^{(2s+1)m}.\]

We split this sum into two. First
\[ \sum_{m=1}^{\lfloor \sqrt{n}\rfloor}  \left(\frac{m}{2\mu n}\right)^{(2s+1)m}\le \sqrt{n} \left(\frac{1}{2\mu \sqrt{n} }\right)^{(2s+1)}=O(n^{-s}). \]
Then
\[ \sum_{m=\lfloor \sqrt{n}\rfloor+1}^{\lfloor\mu n\rfloor} \left(\frac{m}{2\mu n}\right)^{(2s+1)m}\le \sum_{h=1}^{\infty} \left(\frac{1}2\right)^{\sqrt{n}+h}= 2^{-\sqrt{n}}=O(n^{-s}). \]
Thus, the statement follows.
\end{proof}

\begin{lemma}\label{Lemma17}
Let $\varepsilon>0$. Then for any large enough constant $k$, with probability $1-O(n^{-s})$ the following holds:

\begin{enumerate}[(1)]
    \item \label{part1} For all $E\subset {{[n]}\choose {d+1}}$, we have
    \[|K_n^\cup \cap E|\ge  |K_n^{\Sigma}\cap E^\Sigma|-\varepsilon k n^d.\]
    \item \label{part2} \[(1-\varepsilon)\frac{kn^d(d+1+\ell)}{(d+1)!}\le |K_n^\cup(d)| \le(1+\varepsilon)\frac{kn^d(d+1+\ell)}{(d+1)!}.\]
 \item  \label{part3} For all  $A\subset [n]$, we have
\[\|A\|\ge (1-\varepsilon)\frac{|A|}{(d+1+\ell)n}.\]

\end{enumerate}

\end{lemma}
\begin{proof}
Choose $\eta>0$ small enough. We  may assume that the high  probability event provided by Lemma~\ref{SigmaVSCup} occurs. Then for all $\sigma\in {{[n]}\choose{d}}$, we have
\begin{equation}\label{eqSigmaVSCup}
    |C_\sigma\cap K_n^\cup|\ge |K_n^\Sigma\cap C_\sigma^\Sigma|-s-d.
\end{equation}
Summing these inequalities over $\sigma\in {{[n]}\choose{d}}$ and dividing by $d+1$, we have
\[|K_n^\cup(d)|\ge  |K_n^\Sigma(d)|-\frac{s+d}{d+1}{n \choose d}\ge |K_n^\Sigma (d)|-\eta\frac{kn^d(d+1+\ell)}{(d+1)!}\]
provided that $k$ is large enough. Observing that $  |K_n^{\Sigma}\cap E^\Sigma|-|K_n^\cup \cap E|\le |K_n^{\Sigma}(d)|-|K_n^\cup (d)| $, part~\eqref{part1} follows.

Also, obviously
\[|K_n^\cup(d)|\le | K_n^\Sigma(d)|.\]

Observing that $\mathbb{E} |K_n^\Sigma(d)|=(1+o(1)) \frac{kn^d(d+1+\ell)}{(d+1)!}$, and using the concentration inequality in Lemma~\ref{berstein}, we see that with very large probability, we have

\[\left| |K_n^\Sigma (d)|-\frac{kn^d(d+1+\ell)}{(d+1)!}\right|<\eta \frac{kn^d(d+1+\ell)}{(d+1)!}.\]
Thus, on the intersection of the events above, we have
\[(1+\eta)\frac{kn^d(d+1+\ell)}{(d+1)!}\ge |K_n^\cup|\ge (1-2\eta)\frac{kn^d(d+1+\ell)}{(d+1)!}.\]
Thus, part~\eqref{part2} follows.

The number of $d$ dimensional faces of $K_n^\cup$ containing a given $v\in [n]$ can be expressed as
\[\frac{1}d\sum_{\substack{\sigma\in {{[n]}\choose{d}}\\v\in \sigma }} |K_n^\cup\cap C_{\sigma}|\]

By Lemma~\ref{pure}, $T_{n,d,\ell}$ is pure $d$-dimensional complex with complete $d-1$-skeleton, so $ T_{n,d,\ell} \cap C_{\sigma}$ is non-empty. Combining this with \eqref{eqSigmaVSCup}, we get that on the event provided by Lemma~\ref{SigmaVSCup},
\[\frac{1}d\sum_{\substack{\sigma\in {{[n]}\choose{d}}\\v\in \sigma }} |K_n^\cup\cap C_{\sigma}|\ge \frac{1}d\sum_{\substack{\sigma\in {{[n]}\choose{d}}\\v\in \sigma }} \left(|K_n^\Sigma \cap C_{\sigma}^\Sigma|-d-s\right)\ge \frac{1}d{{n-1}\choose{d-1}} (k-d-s)\ge (1-\eta) \frac{kn^{d-1}}{d!}\]
provided that $k$ is large enough.

Thus, assuming that the event in part \eqref{part2} also occurs, for any $A\subset [n]$, we have
\[\|A\|\ge \frac{(1-\eta) \frac{kn^{d-1}}{d!}|A|}{(d+1)|K_n^\cup|}\ge \frac{(1-\eta) \frac{kn^{d-1}}{d!}|A|}{(d+1)(1+\eta)\frac{kn^d(d+1+\ell)}{(d+1)!}} \ge (1-\varepsilon)\frac{|A|}{(d+1+\ell)n}.\]
provided that $\eta$ is small enough, which gives part~\eqref{part3}.
\end{proof}

\begin{lemma}\label{skeletonpart1}
Assume that $k$ is large enough. With probability at least $1-O(n^{-s})$, the following holds: for all $A\subset [n]$ such that $0<|A|<\mu n$, we have
$\|E(A,A)\|< 4\alpha \|A\|$.

\end{lemma}
\begin{proof}
Choose $\varepsilon=\frac{1}2$. Assuming that $k$ is large enough, we can restrict our attention to the case when the events in Lemma~\ref{Lemma16} and Lemma~\ref{Lemma17} both occur. Then for $0<|A|<\mu n$, we have

\begin{multline*}
    \|E(A,A)\|\le \frac{{{d+1}\choose{2}}|E_{\ge 2}^d(A)\cap K_n^\cup|}{{{d+1}\choose{2}}|K_n^\cup|}\le \frac{ k\alpha_0 |A|n^{d-1}}{(1-\varepsilon)\frac{kn^d(d+1+\ell)}{(d+1)!}}\\ \le \frac{\alpha}{1-\varepsilon}\frac{|A|}{(d+1+\ell)n}\le \frac{\alpha}{(1-\varepsilon)^2}\|A\|\le 4\alpha\|A\|.
\end{multline*}
\end{proof}

\begin{lemma}\label{skeletonpart2}
Assume that $k$ is large enough. With probability at least $1-O(n^{-s})$, the following holds: for all $A\subset [n]$ such that $\mu n\le |A|$, we have
$\|E(A,A)\|\le 4 \|A\|^2$.
\end{lemma}
\begin{proof}
Let \[E^d_i(A)=\left\{\sigma\in {{n}\choose{d+1}}\,:\,|\sigma\cap A|= i\right\}.\]

With the notation $m=|A|$, we have

\[|E^d_i(A)|={{m}\choose{i}}{{n-m}\choose{d+1-i}}.\]

Note that we have a $c>0$ such that for $\mu n\le |A|\le\frac{3}{4}n$, and $0\le i\le d+1$, we have
\[\mathbb{E}|E^d_i(A)\cap T_{n,d,\ell}|\ge cn^{d}\ge cn,\text{ and thus, }\mathbb{E}|E^d_i(A)^\Sigma\cap T_{n,d,\ell}^\Sigma|\ge  kcn.\]

Thus, it follows from Lemma~\ref{berstein} that for $\eta=\frac{1}6$, we have

\begin{align}\mathbb{P}\Big(\left||E^d_i(A)^\Sigma\cap K_n^\Sigma|-\mathbb{E}|E^d_i(A)^\Sigma\cap K_n^{\Sigma}|\right|&\le \frac{\eta}2 \mathbb{E}|E^d_i(A)^\Sigma\cap K_n^{\Sigma}|\Big)\nonumber\\&\ge 1-2\exp\left(-\frac{\eta^2ckn}{16}\right)\nonumber \\&\ge 1-2\cdot 4^{-n}\label{bottleneck}
\end{align}
provided that $k$ is large enough.

Therefore, with probability at least $1-2^{1-n}$, we have
\begin{equation}\label{eqEdiBern}
    \left||E^d_i(A)^\Sigma\cap K_n^{\Sigma}|-\mathbb{E}|E^d_i(A)^\Sigma\cap K_n^{\Sigma}|\right|\le \frac{\eta}2 \mathbb{E}|E^d_i(A)^\Sigma\cap K_n^{\Sigma}|
\end{equation}
for all $\mu n\le |A|\le \frac{3}4n$. Applying part~\eqref{part1} of Lemma~\ref{Lemma17} with the choice of $\varepsilon=\frac{c\eta}2$, it follows that with probability at least $1-O(n^{-s})$, we have

\[\left||E^d_i(A)\cap K_n^\cup|-|E^d_i(A)^\Sigma\cap K_n^{\Sigma}|\right|\le \frac{\eta}2\mathbb{E}|E^d_i(A)^\Sigma\cap K_n^\Sigma|=\frac{\eta}2\frac{k(d+1+\ell)}{n+\ell} |E^d_i(A)|\]
for all $\mu n\le |A|\le \frac{3}4n$.

Combining this with \eqref{eqEdiBern}, we see that with probability at least $1-O(n^{-s})$, we have
\[(1-\eta)\frac{k(d+1+\ell)}{n+\ell} |E^d_i(A)|\le |E^d_i(A)\cap K_n^\cup|\le (1+\eta)\frac{k(d+1+\ell)}{n+\ell} |E^d_i(A)|\]
for all $\mu n\le |A|\le \frac{3}4n$.

Note that on this event
\begin{align*}\|E(A,A)\|{{d+1}\choose 2}|K_n^\cup(d)|&=\sum_{i=2}^{d+1} {{i}\choose {2}} |E^d_i(A)\cap K_n^\cup|\\&\le (1+\eta) \frac{k(d+1+\ell)}{n+\ell} \sum_{i=2}^{d+1} {{i}\choose {2}}  {{m}\choose{i}}{{n-m}\choose{d+1-i}}\\&=(1+\eta)\frac{k(d+1+\ell)}{n+\ell}{{m}\choose{2}}{{n-2}\choose{d-1}}, \end{align*}
and
\begin{align}\label{Anormbecs}\|A\|(d+1)|K_n^\cup(d)|&=\sum_{i=1}^{d+1} i |E^d_i(A)\cap K_n^\cup|\\&\ge (1-\eta)\frac{k(d+1+\ell)}{n+\ell}  \sum_{i=1}^{d+1} i  {{m}\choose{i}}{{n-m}\choose{d+1-i}}\nonumber\\&=(1-\eta)\frac{k(d+1+\ell)}{n+\ell} m{{n-1}\choose{d}}.\nonumber \end{align}

By part~\eqref{part2} of Lemma~\ref{Lemma17}, we have \begin{equation}\label{eqcorLemma17}
    (1-\eta)\frac{k(d+1+\ell)}{n+\ell}{n\choose{d+1}}\le|K_n^\cup(d)|\le (1+\eta)\frac{k(d+1+\ell)}{n+\ell}{n\choose{d+1}}
\end{equation}
with probability at least $1-O(n^{-s}).$

Thus, on the intersection of all the events above, for all $\mu n\le |A|\le \frac{3}4n$, we have

\begin{align*}\frac{\|E(A,A)\|}{\|A\|^2}&=\frac{2(d+1)}d \frac{\left(\|E(A,A)\|{{d+1}\choose 2}|K_n^\cup(d)|\right) |K_n^\cup(d)|}{\left(\|A\|(d+1)|K_n^\cup(d)|\right)^2}\\&\le \left(\frac{1+\eta}{1-\eta}\right)^2\frac{2(d+1)}d \frac{{{m}\choose{2}}{{n-2}\choose{d-1}}{{n}\choose{d+1}}}{\left(m{{n-1}\choose{d}}\right)^2}\\&=\left(\frac{1+\eta}{1-\eta}\right)^2 \frac{(m-1)n}{m(n-1)}\\&\le \left(\frac{1+\eta}{1-\eta}\right)^2\le 4.
\end{align*}

By \eqref{Anormbecs} and \eqref{eqcorLemma17}, on the intersection of all the events above, for all $|A|=\frac{3}4n$, we have

\[\|A\|\ge \frac{(1-\eta)}{1+\eta} \frac{\frac{3}4n{{n-1}\choose{d}}}{(d+1){{n}\choose{d+1}}}\ge \frac{1}2\]
for any large enough $n$. Obviously, then we have $\|A\|\ge \frac{1}2$ for any  $|A|\ge \frac{3}4n$. Thus, in this case
$\|E(A,A)\|\le 1\le 4 \|A\|^2.$
\end{proof}

Lemma~\ref{lemmaskeleton} follows by combining Lemma~\ref{skeletonpart1} and Lemma~\ref{skeletonpart2}.

\begin{remark}\label{remarkbottleneck}
    Inequality \eqref{bottleneck} requires that $k=\Omega(c^{-1})$, which by the choice of $c$ means that $k=\Omega(\alpha^{-2})$. Typically this is the bottleneck for the choice of $k$.
\end{remark}

\section{The proof of Theorem~\ref{thm1}}\label{secfinal}

\begin{lemma}\label{NormVSSize}
Let $d\ge 1$ and $\ell\ge 0$ be integers, and let $\mu,s>0$. Consider $K_n^\cup=T_{n,d,\ell}^\cup$. Let $\mu'=\frac{\mu}{4e(d+1)!}$. For any large enough constant $k$, with probability at least $1-O(n^{-s})$ the following event occurs: For any $0\le i\le d-1$ and a minimal $A\in C^i(K_n^\cup,\mathbb{F}_2)$ such that $\|A\|_{K_n^\cup}\ge \mu$, we have
\[\min_{B\in A+B^i(\Delta_{[n]},\mathbb{F}_2)} |B|\ge \mu'n^{i+1}.\]
\end{lemma}
\begin{proof}
Assume that for a minimal $A\in C^i(K_n^\cup,\mathbb{F}_2)$ such that $\|A\|_{K_n^\cup}\ge \mu$, we have
\[\min_{B\in A+B^i(\Delta_{[n]},\mathbb{F}_2)} |B|< \mu'n^{i+1}.\] Let $B=\argmin_{B\in A+B^i(\Delta_{[n]},\mathbb{F}_2)} |B|$. Then, on one hand $|B|<\mu'n^{i+1}$. On the other hand, since $A$ was minimal, $\|B\|_{K_n^\cup}\ge \|A\|_{K_n^\cup}\ge \mu$.

Thus, it is enough to prove that if $k$ is large enough, then with probability at least $1-O(n^{-s})$ the following event occur:\[\text{ For any $0\le i\le d-1$ and  $B\in C^i(K_n^\cup,\mathbb{F}_2)$ such that $|B|< \mu' n^{i+1}$, we have $\|B\|_{K_n^\cup}<\mu$.}\]

For a $B$ like above, let \[E_B=\left\{\sigma\in {{[n]}\choose{d+1}}\,:\, \tau\subset \sigma \text{ for some }\tau\in B\right\}.\]
Then $|E_B|\le |B|n^{d-i}\le \mu'n^{d+1}$. Then the argument of Lemma~\ref{Lemma15} gives us that for $\kappa=2e\mu'$, we have
\[\mathbb{P}(|E_B\cap K_n^\cup|\ge \kappa k n^{d})\le {{\mu'n^{d+1}}\choose {\kappa k n^d}} \left(\frac{k(d+1+\ell)}{n+\ell}\right)^{\kappa k n^d}\le \left(\frac{e\mu'}{\kappa}\right)^{\kappa k n^d}  \le 2^{-\kappa k n^d}.\]

The number of possible choices of $B$ is at most $d2^{n^d}$. Thus, by the union bound, if $k$ is large enough then with probability at least $1-O(n^{-s})$, we have $|E_B\cap K_n^\cup|<\kappa k n^{d}$ for all $0\le i\le d-1$ and  $B\in C^i(K_n^\cup,\mathbb{F}_2)$ such that $|B|< \mu' n^{i+1}$. With the choice of $\varepsilon=\frac{1}2$, on the event provided by part~\eqref{part2} of Lemma~\ref{Lemma17}, we have
\[\|B\|\le \frac{{{d+1}\choose{i+1}}|E\cap K_n^\cup|}{{{d+1}\choose{i+1}}|K_n^\cup(d)|}<\frac{\kappa k n^{d}}{\frac{kn^d}{2(d+1)!}}=2\kappa(d+1)!=\mu.\]
\end{proof}

The next theorem is clearly stronger than Theorem~\ref{thm1}.
\begin{theorem}\label{thmstronger}
    Let $d\ge 1$ and $\ell\ge 0$ be integers, and let $s>0$. Then there is a $\beta=\beta(d,\ell,s)>0$ such that for all large enough constant $k$, we have
    \[\mathbb{P}(\mathfrak{h}(T_{n,d,\ell}^\cup)<\beta)=O(n^{-s}).\]
\end{theorem}
\begin{proof}
We prove by induction on $d$. Assume that we proved the statement for all $d',\ell',s'$ such that $d'<d$.

Let
\[\beta=\min_{i=1}^{d-1} \beta(d-i,\ell+i,s+i)\]
if $d>1$, for $d=1$, we can set $\beta=1$.

Lemma~\ref{Lemmalinks} says that for any $1\le i\le d-1$ and $\sigma\in {{[n]}\choose{i}}$, the link $\Lk(\sigma,T_{n,d,\ell}^\cup)$ has the same distribution as $T_{n-i,d-i,\ell+i}^\cup$. Thus, by the induction hypothesis and the choice of $\beta$, we have
\[\mathbb{P}(\mathfrak{h}(\Lk(\sigma,T_{n,d,\ell}^\cup))<\beta)=O(n^{-(s+i)})\]
provided that $k$ is large enough.

Then by the union bound, we get that with probability at least $1-O(n^{-s})$, the complex $T_{n,d,\ell}^\cup$ satisfies $(a)$ of Theorem~\ref{evrakaufman} , that is, for all $\sigma \in T_{n,d,\ell}^\cup$ such that $1\le |\sigma|\le d-1$, we have
    \[\mathfrak{h}(\Lk(\sigma,T_{n,d,\ell}^\cup))\ge \beta.\]

Let $\alpha=\alpha(d,\beta)$ be the constant provided by Theorem~\ref{evrakaufman}. In a similar manner as above, for all large enough $k$, we can combine Lemma~\ref{lemmaskeleton} and the union bound to obtain that with probability at least $1-O(n^{-s})$, the complex $T_{n,d,\ell}^\cup$ satisfies $(b)$ of Theorem~\ref{evrakaufman} , that is, for all $\sigma \in T_{n,d,\ell}^\cup$ such that $0\le |\sigma|\le d-1$, the link $\Lk(\sigma,T_{n,d,\ell}^\cup)$ is an $\alpha$-skeleton expander.

Thus, on the intersection of the events above (which has probability at least $1-O(n^{-s})$),  Theorem~\ref{evrakaufman} can be applied to give that for some $\bar{\varepsilon},\bar{\mu}>0$ the following holds:

For any $0\le i\le d-1$ and a minimal $A\subset C^i(T_{n,d,\ell},\mathbb{F}_2)$ such that $\|A\|\le \bar{\mu}$, we have

\begin{equation}\label{kicsiexpansion}\|\delta_{i,\mathbb{F}_2} A\|\ge \bar{\varepsilon}\|A\|. \end{equation}

Next we prove that with  probability at least $1-O(n^{-s})$ the following holds:

For any $0\le i\le d-1$ and a minimal $A\subset C^i(T_{n,d,\ell},\mathbb{F}_2)$ such that $\|A\|\ge \bar{\mu}$, we have

\begin{equation}\label{nagyexpansion}\|\delta_{i,\mathbb{F}_2} A\|\ge \beta' \|A\|, \end{equation}
where $\beta'$ will be specified later.

To prove this, we can restrict our attention to the high probability event provided by Lemma~\ref{NormVSSize}.

On this event, if $A\subset C^i(T_{n,d,\ell})$ is minimal such that $\|A\|\ge \bar{\mu}$, then for \[B=\argmin_{B\in A+B^i(\Delta_{[n]},\mathbb{F}_2)} |B|,\] we have $\|\delta_{i,\mathbb{F}_2} A\|=\|\delta_{i,\mathbb{F}_2} B\|$, $B$ is $\Delta_{[n]}$-minimal and $|B|\ge \bar{\mu}' n^{i+1}$ for $\bar{\mu}'=\frac{\bar{\mu}}{4e(d+1)!}$.

Then to show \eqref{nagyexpansion} holds with high probability, it is enough to show that with  probability at least $1-O(n^{-s})$ the following holds:

For any $0\le i\le d-1$ and a $\Delta_{[n]}$-minimal $B\subset C^i(\Delta_{[n]})$ such that $|B|\ge \bar{\mu}'n^{i+1}$, we have

\begin{equation}\label{deltaB}\|\delta_{i,\mathbb{F}_2}^{T_{n,d,\ell}} B\|\ge \beta'. \end{equation}

Let $D=\delta_{i,\mathbb{F}_2}^{\Delta_{[n]}} B$. By Lemma~\ref{Deltaexpansion} , we see that

\[|D|\ge \frac{n|B|}{i+2}.\]

Let $E$ be all the elements of ${[n]}\choose{d+1}$ which contains an element of $D$, then
\[|E|\ge \frac{{{n-i-2}\choose {d-i-1}}|D|}{{{d+1}\choose{i+2}}}\ge \frac{n{{n-i-2}\choose {d-i-1}}|B|}{(i+2){{d+1}\choose{i+2}}}\ge  \frac{n{{n-i-2}\choose {d-i-1}}\bar{\mu}' n^{i+1}}{(i+2){{d+1}\choose{i+2}}}\ge  c n^{d+1}  \]
for some constant $c=c(d,\ell,s)>0$ and all large enough $n$.

Note that
\[\mathbb{E}|T_{n,d,\ell}^\Sigma \cap E^\Sigma|\ge k c n^d\]
for all large enough $n$.

Thus, by Lemma~\ref{berstein}, we  have

\begin{align*}\mathbb{P}\left(|T_{n,d,\ell}^\Sigma \cap E^\Sigma|\le \frac{k c n^d}2\right)&\le \mathbb{P}\left(|T_{n,d,\ell}^\Sigma \cap E^\Sigma|\le \frac{1}2 \mathbb{E}|T_{n,d,\ell}^\Sigma \cap E^\Sigma|\right)\\&\le 2\exp\left(-\frac{1}{16}\mathbb{E}|T_{n,d,\ell}^\Sigma \cap E^\Sigma|\right)\\&\le 2\exp\left(-\frac{kcn^d}{16}\right). \end{align*}

Clearly,
\[\sum_{i=0}^{d-1}|C^i(\Delta_{[n]},\mathbb{F}_2)|\le d2^{n^d}.\]

Provided that $k$ is large enough, 
\begin{equation}
2^{n^d}\exp\left(-\frac{kcn^d}{16}\right)=O(n^{-s}).\label{bottleneck2}
\end{equation}
Thus, it follows from the union bound that with probability $1-O(n^{-s})$, we have
\[|T_{n,d,\ell}^\Sigma \cap E^\Sigma|\ge \frac{k c n^d}2\]
for all choices of $B$ like above. If the high probability event provided part~\eqref{part1} of  Lemma~\ref{Lemma17} occur with the choice $\varepsilon=\frac{c}4$, then

\[|T_{n,d,\ell}^\cup\cap E|\ge \frac{k c n^d}4.\]

Furthermore, on the high probability event provided by part~\eqref{part2} of Lemma~\ref{Lemma17}, we have $|T_{n,d,\ell}^\cup(d)|\le 2k n^d(d+1+\ell)$. Thus, on the intersection of all the event above

\[\|\delta_{i,\mathbb{F}_2}^{T_{n,d,\ell}} B\|\ge \frac{|T_{n,d,\ell}^\cup\cap E|}{{{d+1}\choose{i+2}}|T_{n,d,\ell}^\cup|}\ge \frac{\frac{k c n^d}4}{{{d+1}\choose{i+2}}2k n^d(d+1+\ell)}\ge \beta'\]
for the choice of $\beta'=\frac{c}{8\cdot 2^{d+1}}$.

Thus, \eqref{deltaB} follows.

Choosing $\beta(d,\ell,s)=\min(\beta',\bar{\varepsilon})$, and combining \eqref{kicsiexpansion}, \eqref{nagyexpansion} with \eqref{hiseconddef}, the theorem follows.

\end{proof}

\begin{remark}\label{remarkdouble2}
For the choice of $\alpha$ given by Evra and Kaufmann, $\alpha^{-1}$ is at least double exponential in $d$. Combining this with Remark~\ref{remarkbottleneck}, we see that the $k_d$ provided by our proof is at least double exponential in $d$. Note that \eqref{bottleneck2} also requires a double exponential bound on $k_d.$ Evra and Kaufman provide explicit expression for the constant $\alpha$, $\bar{\varepsilon}$ and $\bar{\mu}$ in Theorem~\ref{evrakaufman}. Thus, with some effort one can provide an explicit bound on $k_d$ in our proof. However, as we do not expect this constant to be optimal, we did not pursue this. 
\end{remark}

\bibliography{references}
\bibliographystyle{plain}

\bigskip

\noindent Andr\'as M\'esz\'aros, \\
HUN-REN Alfr\'ed R\'enyi Institute of Mathematics, \\Budapest, Hungary,\\ {\tt meszaros@renyi.hu}

\end{document}